\documentclass[acmsmall,screen]{MapleTrans}

\usepackage[normalem]{ulem}

\usepackage{enumitem}
\usepackage{hyperref} 
\usepackage{etoolbox}
\usepackage{graphicx}
\usepackage{amsthm}
\usepackage{amsmath}
\usepackage{amsfonts}

\usepackage{breqn}
\usepackage{lineno}
\usepackage{maple}
\usepackage{xfrac}
\usepackage{xcolor}
\definecolor{mygreen}{rgb}{0,0.6,0}
\definecolor{mygray}{rgb}{0.5,0.5,0.5}
\definecolor{mymauve}{rgb}{0.58,0,0.82}
\definecolor{altblue}{rgb}{0.0,0.6,1.0}
\definecolor{lstbg}{cmyk}{0.05, 0.01, 0, 0}
\definecolor{morebluish}{cmyk}{0.06,0.04,0,0}

\usepackage{listings}
\input{listings-maple-definition.sty}
\setcopyright{acmlicensed}
\copyrightyear{2025}
\acmYear{2025}
\acmDOI{10.5206/mt.v5i1.16090}

\acmJournal{MAPLETRANS}
\acmVolume{5}
\acmNumber{1}
\acmArticle{16090}
\acmMonth{2}

\begin{document}

\title{Series and Product Representations of Gamma, Pseudogamma and Inverse Gamma Functions}

\author{David Peter Hadrian Ulgenes}

\email{dpulgene@uio.no}
\orcid{0009-0002-6649-4210}

\affiliation{%
  \institution{University of Oslo}
  \streetaddress{Sem Sælands vei 24}
  \city{Oslo}
  \country{Norway}
  \postcode{0371}
}

\begin{teaserfigure}
 \centering
\includegraphics[width=10cm]{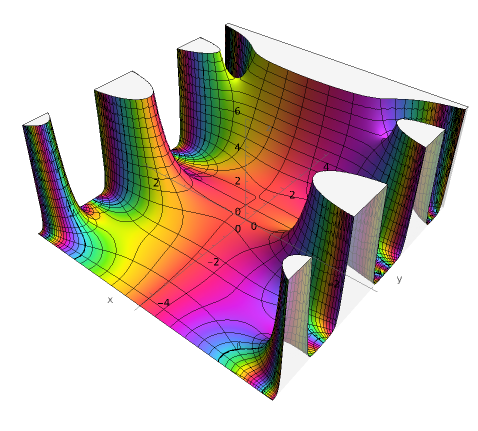}
\caption{Complex plot of the $\Lambda$ function. Code available in appendix.}
\end{teaserfigure}
\renewcommand{\shortauthors}{Ulgenes}

\begin{abstract}
\textbf{Abstract.} We derive product and series representations of the gamma function using Newton interpolation series. Using these identities, a new formula for the coefficients in the Taylor series of the reciprocal gamma function is found. We also find two new series representations for the Euler-Mascheroni constant, containing only rational terms. After that, we introduce a new pseudogamma function which we call the $\Lambda$ function. This function interpolates the factorial
at the positive integers, the reciprocal factorial at the negative integers, and is convergent for the entire real axis.
Finally, we conjecture a novel series representation for the principal branch of the inverse gamma function $\text{inv}\Gamma_0(z)$.

\end{abstract}



\maketitle

\section{Introduction}
The gamma function is the most widely used extension of the factorial function to the complex plane. The gamma function is defined for $x > 0$ by Euler’s second integral
$$\Gamma(x)=\int_{0}^{\infty}e^{-t}t^{x-1}dt,$$ 
from which the functional equation $$\Gamma(x+1)=x\Gamma(x)$$ follows by integrating by parts.

Consider the problem of finding an interpolating polynomial to compute the gamma function. Since $\Gamma(n)=(n-1)!$, it would be convenient to consider a series of the form 

\begin{equation}\label{eq:234}
\Gamma\left(x+1\right)=a_{0}+a_{1}x+a_{2}x\left(x-1\right)+a_{3}x\left(x-1\right)\left(x-2\right)+...
\end{equation}
where the coefficients $a_0, a_1, ...$ are iteratively computed by setting $x=0, 1, 2...$ and solving for $a_n$. Performing this calculation, we obtain the first few coefficients: $1,\ 0,\ \frac{1}{2},\ \frac{1}{3},\ \frac{3}{8}$, etc.

The Scottish mathematician James Stirling was the first to do this \cite{Historical}. However, the resulting series diverges everywhere except at the nonnegative integers, which Stirling himself pointed out. This is because the factorials grow too rapidly to be suitable as nodes for interpolation.

Later, in 1901, Hermite \cite{Hermite} considered the series 

\begin{equation}\label{eq:235}
\ln\Gamma\left(x+1\right)=a_{0}+a_{1}x+a_{2}x\left(x-1\right)+a_{3}x\left(x-1\right)\left(x-2\right)+...,
\end{equation}
where he was able to find the general expression for the coefficients $a_n=-\frac{1}{n!}\int_{0}^{\infty}\frac{e^{-t}}{t}\left(e^{-t}-1\right)^{n-1}dt$ for integer $n>1$. Hermite was also able to show that the series converges for complex $z$ with real component greater than 0. 

In this paper, we aim to take the ideas of Stirling and Hermite further by finding additional representations of the gamma function using interpolation series. We will make use of Newton's interpolation series formula

\begin{equation}\label{eq:1}
f(x)=\sum_{n=0}^{\infty}\binom{x}{n}\left(-1\right)^{n}\sum_{k=0}^{n}\left(-1\right)^{k}\binom{n}{k}f(k)
\end{equation}

(which equations \ref{eq:234} and \ref{eq:235} are special cases of) to find identities, which we will then show to converge. 
Using these identities, we find a new definition of the Taylor series of $1/\Gamma(x+1)$, as well as several rational series for the Euler-Mascheroni constant.

After that, we take Newton interpolation and the factorial a step further, by showing how Newton series can be used to define pseudogamma functions. A pseudogamma function is any function apart from $\Gamma$ which also interpolates the factorial. Unlike $\Gamma$, the pseudogamma function introduced in this paper (which we call $\Lambda$) converges for the entire real axis and has several interesting properties, for instance, the reflection formula $\Lambda(x)\Lambda(-x)=1$.

Finally, we propose how Newton series can be used to compute the principal branch of the inverse gamma function $\Gamma(y)=x$ (where we adopt the notation $\text{inv}\Gamma_0(z)$ from \cite{notation}). Despite a surge in research interest in recent years (see for instance \cite{1} \cite{2} \cite{3}), very little is known about the inverse gamma function.

\section{Main Results}

\begin{theorem}\label{theorem:6}
For $x>0$,\footnote{This equation has been mentioned before, for instance in \cite{Norlund}.}
\begin{equation}\label{eq:2}\frac{1}{\Gamma\left(x+1\right)}=\sum_{n=0}^{\infty}\left(-1\right)^{n}\binom{x}{n} L_{n}\left(1\right)\end{equation} 
where $L_{n}\left(1\right)$ is the Laguerre polynomial $L_n(x)=\sum_{k=0}^n \binom{n}{k}\frac{(-1)^k}{k!} x^k $ \cite{Laguerre} at $x=1$.

\end{theorem}
\begin{proof}
  Equation \ref{eq:2} follows from computing the Newton series of $1/\Gamma(x+1)$ and writing it using the definition of the Laguerre polynomial.
For convergence, we note that 
the generalized Laguerre Polynomial $L_{n}^{\left(\alpha\right)}\left(x\right)$ (satisfying $L_{n}^{\left(0\right)}\left(x\right)=L_{n}\left(x\right)$)
has the following bound \cite{Bound}
$$\left|L_{n}^{\left(\alpha\right)}\left(x\right)\right|\le\frac{\left(\alpha+1\right)_{n}}{n!}e^{\frac{x}{2}},\: \alpha\ge0,\: x\ge 0, \:n=0, 1, 2,...$$

where $\left(x\right)_{n}$ is the Pochhammer symbol.
Hence equation \ref{eq:2} converges where 
\begin{equation}\label{eq:9}
 \sum_{n=0}^{\infty}\left|\binom{x}{n}\right|
\end{equation}

does. But writing the summands in \ref{eq:9} as $\left|\frac{a_{n+1}}{a_{n}}\right|$ we obtain $\left|\frac{x-n}{n+1}\right|$ which for large n and positive $x$ gives $\frac{n-x}{n+1}$. Expanding this at $n=\infty$ gives $1-\frac{x+1}{n}+O\left(n^{-2}\right)$ hence equation \ref{eq:2} converges at least for $x>0$ by the Gauß test \cite{Gauss}.
\end{proof}

\begin{theorem}
The reciprocal gamma function can be expressed using the rational Taylor (Maclaurin) power series

\begin{equation}\label{eq:3}
\frac{1}{\Gamma\left(x+1\right)}=\sum_{k=0}^{\infty}a_{k}x^{k}
\end{equation}

\begin{equation}\label{eq:4}
=1+\gamma x+\frac{1}{12}\left(6\gamma^2-\pi^2\right)x^2+...
\end{equation}
where $$a_{k}=\sum_{n=0}^{\infty}\left(-1\right)^{n}s\left(n, k\right)\frac{L_{n}\left(1\right)}{n!},$$and $s\left(n, k\right)$ are the signed Stirling numbers of the first kind.

\end{theorem}

\begin{proof}

Substituting
$\binom{x}{n}=\frac{1}{n!}\sum_{k=0}^{\infty}s\left(n, k\right)x^{k}$ \cite{Stirling} in lemma \ref{theorem:6} and interchanging the order of summation (permissible by the uniform convergence of equation \ref{eq:2}, which was already demonstrated) yields the desired result. Here we use the convention that $s(n, k)=0$ for $k>n$.\end{proof}
\begin{corollary}
The equation $$f_m(x)=\sum_{k=0}^{\infty}x^{k}\sum_{n=0}^{m}\left(-1\right)^{n}s\left(n,\ k\right)\frac{L_{n}\left(1\right)}{n!}$$
for $m=1, 2, 3,\:...$ generates rational approximations of the gamma function. For instance, for $m=3, 4$ respectively,
\begin{equation}\label{eq:10}
\frac{1}{\Gamma\left(x+1\right)}\approx f_3(x)= 1+\frac{17}{36}x-\frac{7}{12}x^{2}+\frac{1}{9}x^{3},
\end{equation}
and 
\begin{equation}\label{eq:11}\frac{1}{\Gamma\left(x+1\right)}\approx f_4(x)=1+\frac{181}{288}x-\frac{167}{192}x^{2}+\frac{77}{288}x^{3}-\frac{5}{192}x^{4}.
\end{equation}
Computing the Taylor series of the reciprocals of equations \ref{eq:10}, \ref{eq:11} yields rational approximations to $\Gamma$.\footnote{The denominator of $x^n$ appears to equal $a^n$ (in the example, $a=36, 288$ respectively). Perhaps a closed form expression exists for these coefficients.} For instance, 
\begin{equation}\label{eq:555}
\Gamma(x+1)\approx \frac{1}{f_3(x)}=1-\frac{17}{36}x+\frac{1045}{1296}x^2-\frac{35801}{46656}x^3+\hdots,
\end{equation}
\begin{equation}\label{eq:556}
\Gamma(x+1)\approx \frac{1}{f_4(x)}=1-\frac{181}{288}x+\frac{104905}{82944}x^2-\frac{38432557}{23887872}x^3+\frac{15859708705}{6879707136}x^4+\hdots.
\end{equation}

\end{corollary}

\begin{figure}[htp]
    \centering
    \fbox{\includegraphics[width=5.5cm]{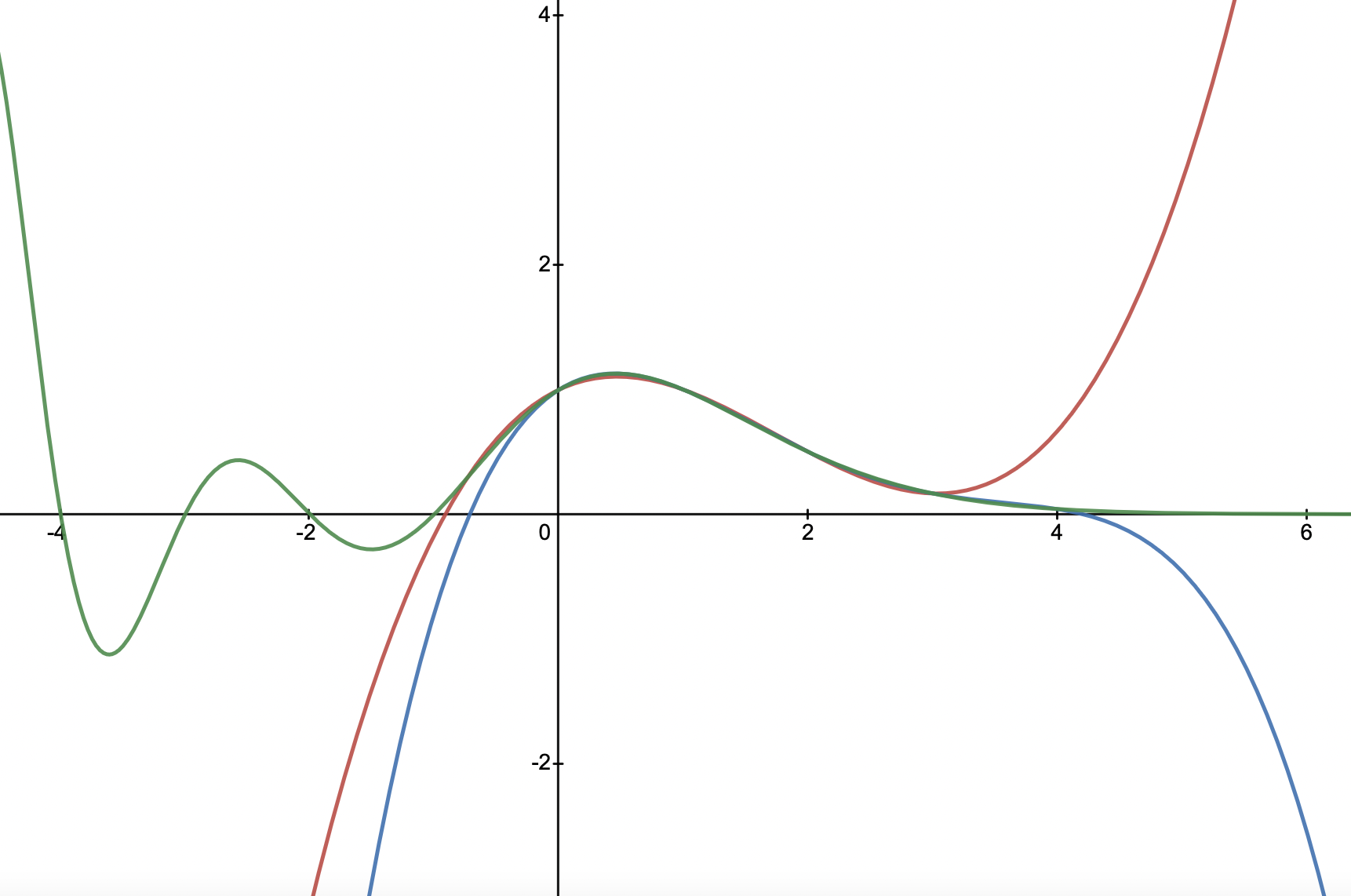}}
    \caption{$1/\Gamma(x+1)$ in green, with \ref{eq:10} and \ref{eq:11} shown in red and blue respectively.}
    \label{fig:Stirling's formula}
\end{figure}

\begin{corollary}
\begin{equation}-\gamma=\sum_{n=1}^{\infty}\frac{L_{n}\left(1\right)}{n}=-\frac{1}{4}-\frac{2}{9}-\frac{5}{32}+...\:.
\end{equation}

\end{corollary}
\begin{proof}
This equation follows from comparing coefficients in equations \ref{eq:3} and \ref{eq:4} and using the identities $s(n, 1)=(-1)^{n-1}(n-1)!$ (also available in \cite{Stirling}) and $s(0, 1)=0$.
\end{proof}

\begin{corollary}
For $x>0 \notin \mathbf{N}$ the reciprocal gamma function can be written as a limit involving the hypergeometric function:\footnote{This is a partial solution to problem 8.6 in \cite{Maple} which asks for what $x$ can the gamma function be written in terms of the hypergeometric function.}
$$\frac{1}{\Gamma\left(x+1\right)}=\lim_{n\to\infty}\left(-1\right)^{n}\binom{x-1}{n}\ _{2}F_{2}\left(-x,\ -n;\ 1,\ 1-x;\ 1\right).$$

\end{corollary}
\begin{proof}
Using induction it is trivial to show that Newton's formula \ref{eq:1} can be written as 
  $$\sum_{n=0}^{N}\binom{x}{n}\left(-1\right)^{n}\sum_{k=0}^{n}\left(-1\right)^{k}\binom{n}{k}f(k)=\sum_{i=0}^{N}f\left(i\right)\sum_{n=i}^{N}\left(-1\right)^{n+i}\binom{x}{n}\binom{n}{i}.$$

Substituting $f(i)=1/i!$ in the formula on the right we obtain, with the help of Maple 

$$\frac{1}{\Gamma(x+1)}=\lim_{N\to\infty}x\binom{x-1}{N}\sum_{i=0}^{N}\frac{\left(-1\right)^{N+i}}{i!(x-i)}\binom{N}{i}$$

which converges for $x>0$ except at the integers.\footnote{The limit has the same value as we have shown that equation \ref{eq:2} converges absolutely.} Rewriting the summation using Maple we obtain the desired equation.
  
\end{proof}

\begin{theorem}\label{theorem:4}
For real values of $x\ge1$ the gamma function has the following Newton series representation

\begin{equation}\label{eq:6}
\Gamma(x)=x^{x-1}\sum_{n=1}^{\infty}\left(-1\right)^{n}\binom{x}{n}\sum_{k=1}^{n}\left(-1\right)^{k}\frac{k!}{k^k}\binom{n}{k}.
\end{equation}

For instance, 
$$\Gamma\left(\frac{3}{2}\right)=\frac{\sqrt{\pi}}{2}=\sqrt{\frac{3}{2}}\left(\frac{3}{2}-\frac{9}{16}-\frac{31}{288}-\frac{517}{12288}+...\right).$$
\end{theorem}

\begin{proof}
Equation \ref{eq:6} is obtained by computing the Newton series of $x^{-x}\Gamma(x+1)$ and multiplying the resulting series by $x^{x-1}$.
For convergence we note that $x^{-x}=\Gamma(x)^{-1}\int_{0}^{\infty}e^{-xt}t^{x-1}dt$ which follows from substituting $u=kt$. Substituting this into equation \ref{eq:6} we obtain, after letting Maple (2021.2) compute the inner sum

$$\Gamma(x)=-x^{x-1}\sum_{n=1}^{\infty}\int_{0}^{\infty}\binom{x}{n}n\frac{\left(e^{-t}t-1\right)^{n}}{e^{t}-t}dt.$$
Taking the absolute values of the summands, one sees that the resulting series would satisfy
\begin{equation}\label{eq:7}
x^{x-1}\sum_{n=1}^{\infty}\int_{0}^{\infty}\left|\binom{x}{n}n\frac{\left(1-e^{-t}t\right)^{n}}{e^{t}-t}\right|dt\le x^{x-1}\sum_{n=1}^{\infty}\int_{0}^{\infty}\left|\binom{x}{n}\frac{n}{e^{t}-t}\right|dt=Cx^{x-1}\sum_{n=1}^{\infty}\left|\binom{x}{n}n\right|,
\end{equation}

where $C=1.35909...\:$. Re-indexed, the last sum can be written, for $x>0$, as $Cx^{x}\sum_{n=0}^{\infty}\left|\binom{x-1}{n}\right|$.
Hence equation \ref{eq:6} converges uniformly at least for real values of $x>1$, which follows directly from the proof of lemma \ref{theorem:6}. But as  $\Gamma(1)=\sum_{n=1}^{\infty}\left(-1\right)^{n}\binom{1}{n}\sum_{k=1}^{n}\left(-1\right)^{k}\frac{k!}{k^k}\binom{n}{k}=1$ the series converges for $x\ge 1$. 
\end{proof}
\begin{corollary}\label{corollary1}
The Euler-Mascheroni constant $\gamma$ has the following series representation containing only rational terms
$$\gamma=-1+\sum_{n=2}^{\infty}\left(n-2\right)!\sum_{k=1}^{n}\frac{\left(-1\right)^{1+k}}{\left(n-k\right)!k^{k}},$$

$$\gamma=-1+\frac{3}{4}+\frac{31}{108}+\frac{517}{3456}+\frac{322537}{3600000}+...  $$
(sequences A360092\cite{OEIS1} and A360091\cite{OEIS2} in the OEIS).
\end{corollary}

\begin{proof}
We begin by differentiating equation \ref{eq:6} after dividing by $x^{x-1}$. So

\begin{equation}\label{eq:22}
\frac{d}{dx} \frac{\Gamma\left(x\right)}{x^{x-1}}=-\frac{\Gamma\left(x\right)}{x^{x}}\left(x+x\ln\left(x\right)-x\psi\left(x\right)-1\right)=\sum_{n=1}^{\infty}\left(-1\right)^{n}\frac{d}{dx} \binom{x}{n}\sum_{k=1}^{n}\left(-1\right)^{k}\frac{k!}{k^k}\binom{n}{k}, 
\end{equation}

where $\psi\left(x\right)=\Gamma^{\prime}\left(x\right)/\Gamma\left(x\right)$ is the digamma function. We are allowed to interchange the sum and the derivative due to the convergence proof in theorem \ref{theorem:4}. To differentiate $\binom{x}{n}$ (for integer \emph{n}) we write

$$
\frac{d}{dx}\binom{x}{n}=\binom{x}{n}\sum_{i=0}^{n-1}\frac{1}{x-i}
$$

which holds by logarithmic differentiation on $\binom{x}{n}$. This can be rewritten as 

\begin{equation}\label{eq:8}
\sum_{i=0}^{n-1}\frac{1}{n!}\left(\prod_{m=0}^{i-1}\left(x-m\right)\right)\prod_{m=i+1}^{n-1}\left(x-m\right)=\sum_{i=0}^{n-1}\frac{\left(-1\right)^{n+i+1}}{n!}\frac{\Gamma\left(x+1\right)}{\Gamma\left(x-i+1\right)}\frac{\Gamma\left(n-x\right)}{\Gamma\left(i-x+1\right)}.
\end{equation}

Extracting the first term in equation \ref{eq:22} and rewriting it using equation \ref{eq:8}  gives

$$
\frac{d}{dx} \frac{\Gamma\left(x\right)}{x^{x-1}}=1+\sum_{n=2}^{\infty}\sum_{k=1}^{n}\frac{\left(-1\right)^{k}}{\left(n-k\right)!k^{k}}\sum_{i=0}^{n-1}\frac{\left(-1\right)^{i-1}\Gamma(x+1)}{\Gamma(x-i+1)}\frac{\Gamma\left(n-x\right)}{\Gamma\left(i-x+1\right)}.
$$

Since $\psi\left(1\right)=-\gamma$, we obtain, at $x=1$:

$$-\gamma=1+\sum_{n=2}^{\infty}\sum_{k=1}^{n}\frac{\left(-1\right)^{k}}{\left(n-k\right)!k^{k}}\sum_{i=0}^{n-1}\frac{\left(-1\right)^{i-1}}{\Gamma\left(2-i\right)}\frac{\Gamma\left(n-1\right)}{\Gamma\left(i\right)}=1+\sum_{n=2}^{\infty}\left(n-2\right)!\sum_{k=1}^{n}\frac{\left(-1\right)^{k}}{\left(n-k\right)!k^{k}}.$$\end{proof}

\begin{theorem}\label{theorem:1}
The gamma function has the following product representation which converges for positive real values of $x$:

\begin{equation}\label{eq:5}
\Gamma(x)  = \frac{1}{x}\prod_{n=1}^{\infty}\prod_{k=1}^{n} k!^{\left(-1\right)^{k+n}\binom{x}{n}\binom{n}{k}}= \frac{1}{x}\left(\frac{2}{1}\right)^{\binom{x}{2}}\left(\frac{3}{4}\right)^{\binom{x}{3}}\left(\frac{32}{27}\right)^{\binom{x}{4}}...\: . 
\end{equation}

\end{theorem}

\begin{proof}
Computing the Newton series of $\ln\Gamma(x+1)$ using equation \ref{eq:1} and exponentiating the result gives the desired formula. 
For convergence we rewrite \ref{eq:5} using the integral 
$\ln \Gamma(x+1)= \int_{0}^{\infty}\left(e^{-xt}-xe^{-t}-1+x\right)/t\left(e^{t}-1\right)dt$. After using Maple to evaluate the inner sum,

$$\ln\Gamma(x+1)=-\sum_{n=2}^{\infty}\int_{0}^{\infty}\binom{x}{n}\frac{e^{-t}}{t}\left(e^{-t}-1\right)^{n-1}dt$$
which Hermite obtained in \cite{Hermite}. For integer n and $t>0$, 
$\frac{e^{-t}}{t}\left(e^{-t}-1\right)^{n-1}\le\ e^{-t}$ and so 
$$\sum_{n=2}^{\infty}\int_{0}^{\infty}\left|\binom{x}{n}\frac{e^{-t}}{t}\left(e^{-t}-1\right)^{n-1}\right|dt\le\sum_{n=2}^{\infty}\int_{0}^{\infty}\left|\binom{x}{n}e^{-t}\right|dt=\sum_{n=2}^{\infty}\left|\binom{x}{n}\right|$$
which we have already shown to converge for $x>0$.
\end{proof}\begin{theorem}
The digamma function has the Newton series
   $$\psi(x+1)=-\gamma-\sum_{k=1}^\infty \frac{(-1)^k}{k} \binom{x}{k}$$
which is commonly called the Stern series.
\end{theorem}
\begin{proof}
For a proof see \cite{Stern}.
\end{proof}

\begin{theorem}\label{theorem:5}
The function 

\begin{equation}\label{eq:54}\Lambda(x)=\prod_{n=1}^{\infty}\prod_{k=1}^{n}k^{\frac{\left(-1\right)^{n+k}}{\left(n-k\right)!\left(k+n-1\right)!}\frac{\left(2k-1\right)}{\left(2n-1\right)}\frac{\Gamma\left(x+n\right)}{\Gamma\left(x+1-n\right)}}
\end{equation}

interpolates the factorial at the positive integers, interpolates the reciprocal factorial at the negative integers, and converges for the entire real axis. \footnote{Note that as, for integer $n$, $\frac{\Gamma\left(x+n\right)}{\Gamma\left(x+1-n\right)}=\left(x-n+1\right)...\left(x+n-1\right)$ the logarithm of equation \ref{eq:54} is also a Newton Series.}

\end{theorem}

\break
\begin{proof}
    
For a positive integer N, we need to show that 

$$
N!=\prod_{n=1}^{N}\prod_{k=1}^{n}k^{\frac{\left(-1\right)^{n+k}\left(2k-1\right)}{\left(n-k\right)!\left(k+n-1\right)!}\frac{\left(N+n-1\right)!}{\left(N-n\right)!\left(2n-1\right)}}
$$

where the product is finite as we take $\frac{1}{\left(N-n\right)!}$ to be 0 for $N>n$. To show the equation holds, we use induction. For the base case $N=1$ we have 
$$1=\prod_{n=1}^{1}\prod_{k=1}^{n}k^{\frac{\left(-1\right)^{n+k}\left(2k-1\right)}{\left(n-k\right)!\left(k+n-1\right)!}\frac{n!}{\left(1-n\right)!\left(2n-1\right)}}$$ which holds. For $N+1$ we have

\begin{equation}\label{eq:29}
\prod_{n=1}^{N+1}\prod_{k=1}^{n}k^{\frac{\left(-1\right)^{n+k}\left(2k-1\right)}{\left(n-k\right)!\left(k+n-1\right)!}\frac{\left(N+n\right)!}{\left(N+1-n\right)!\left(2n-1\right)}}=\left(N+1\right)\prod_{n=1}^{N}\prod_{k=1}^{n}k^{\frac{\left(-1\right)^{n+k}\left(2k-1\right)}{\left(n-k\right)!\left(k+n-1\right)!}\frac{\left(N+n-1\right)!}{\left(N-n\right)!\left(2n-1\right)}},
\end{equation}

where the equality has been established using the factorial identity $(N+1)!=(N+1)N!$. Dividing both sides of equation \ref{eq:29} by the first N terms of the product on the left side and taking natural logarithms gives

\begin{equation}\label{eq:30}
\begin{split}
\sum_{k=1}^{N+1}\frac{\left(-1\right)^{N+1+k}\left(2k-1\right)\left(2N\right)!\ln\left(k\right)}{\left(N+1-k\right)!\left(k+N\right)!}=\ln\left(N+1\right)\\-\sum_{n=1}^{N}\sum_{k=1}^{n}\frac{\left(-1\right)^{n+k}\left(2k-1\right)}{\left(n-k\right)!\left(k+n-1\right)!}\frac{\left(N+n-1\right)!\ln\left(k\right)}{\left(N+1-n\right)!}.
\end{split}
\end{equation}

Taking $\binom{n}{k}=0$ for $k>n$ we can change the upper bound of the \emph{k} summation on the right of equation \ref{eq:30} to \emph{N} without changing the value of the sum. This gives the right side of the equation as
$$
    \ln\left(N+1\right)+\sum_{k=1}^{N}\left(-1\right)^{1+k}\left(2k-1\right)\ln\left(k\right)\sum_{n=1}^{N}\frac{\left(-1\right)^{n}}{\left(n-k\right)!\left(k+n-1\right)!}\frac{\left(N+n-1\right)!}{\left(N+1-n\right)!}.
$$

Maple can evaluate the inner sum: 
\begin{equation}\label{eq:31}
    \sum_{n=1}^{N}\frac{\left(-1\right)^{n}}{\left(n-k\right)!\left(k+n-1\right)!}\frac{\left(N+n-1\right)!}{\left(N+1-n\right)!}\end{equation}$$=\frac{\sin\left(\pi\left(k+1\right)\right)\left(N+k\right)!\left(N+1-k\right)!+\pi\left(-1\right)^{N}\left(2N\right)!\left(N+k\right)\left(N+1-k\right)}{\left(N+k\right)!\left(N+1-k\right)!\pi\left(N+k\right)\left(N+1-k\right)}.
 $$
 
But for integer \emph{k}, $\sin\left(\pi\left(k+1\right)\right)=0$. Disregarding this term equation \ref{eq:31} equals $\frac{\left(-1\right)^{N}\left(2N\right)!}{\left(N+k\right)!\left(N+1-k\right)!}$.
Substituting this into equation \ref{eq:30} and simplifying we see that the inductive step holds. And so for positive integers N the $\Lambda$ function equals the factorial. But since 
$$
    \frac{\Gamma\left(x+n\right)}{\Gamma\left(x+1-n\right)}=\left(x-n+1\right)...\left(x+n-1\right)=x\prod_{k=1}^{n-1}\left(x^{2}-k^{2}\right)
$$

$\Lambda$ satisfies the reflection formula $\Lambda(x)\Lambda(-x)=1$. Hence at the negative integers $\Lambda$ equals the reciprocal factorial.

\begin{figure}
    \centering
    \fbox{\includegraphics[width=5.5cm]{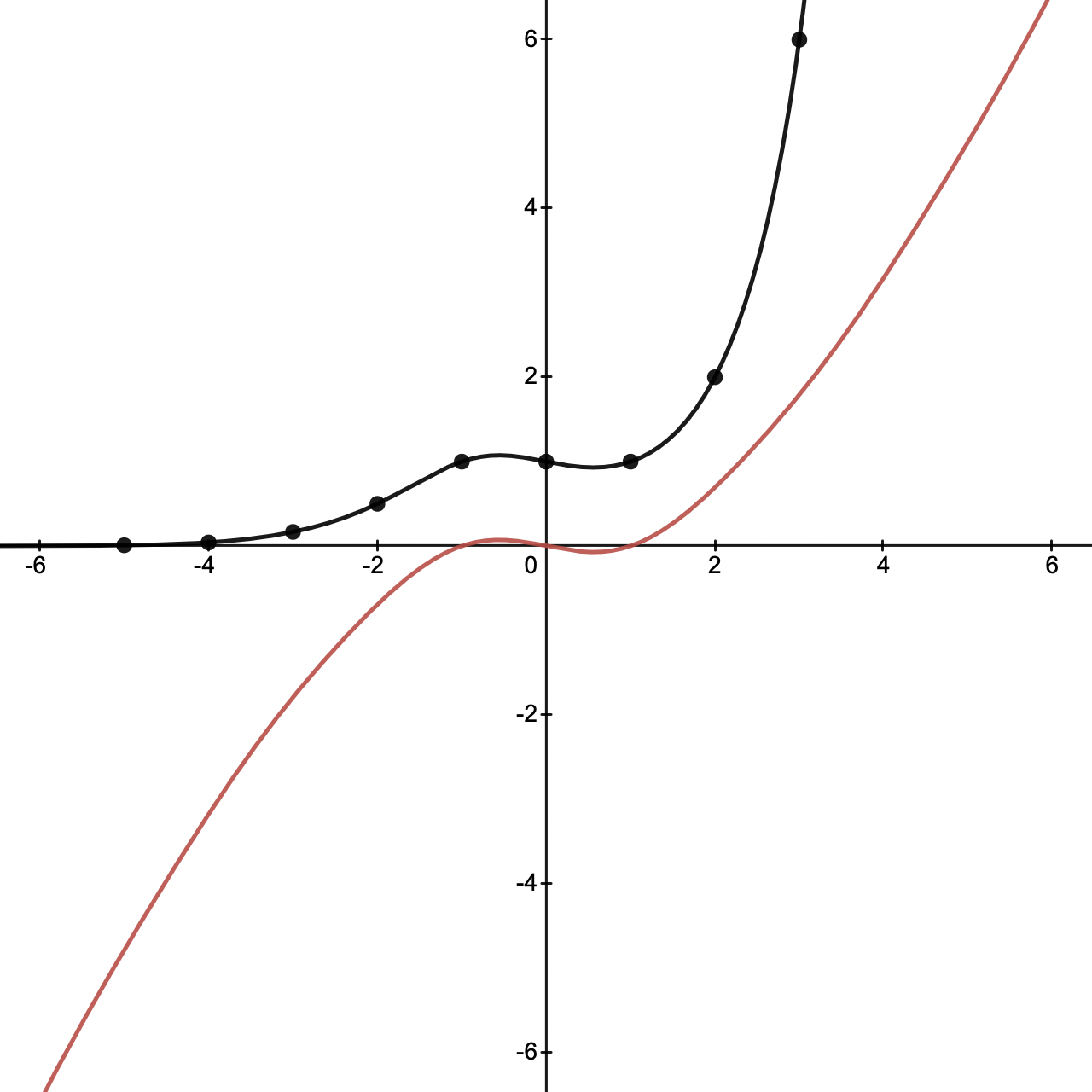}}
    \caption{The graph shows the $\Lambda$ function (black) with the dots representing the values of the factorial at the positive integers and the reciprocal factorial at the negative integers, as well as the log-$\Lambda$ function (red). The image suggests that perhaps the $\Lambda$ function is log-convex for $x>0$ and log-concave for $x<0$. }
    
\end{figure}

The convergence of equation \ref{eq:54} will now be demonstrated. Firstly, taking logarithms gives

\begin{equation}\label{eq:32}
    \ln \Lambda(x)=\sum_{n=1}^{\infty}\frac{\Gamma(x+n)\left(-1\right)^{n}}{\Gamma(x+1-n)\left(2n-1\right)}\sum_{k=1}^{n}\frac{\left(-1\right)^{k}\left(2k-1\right)}{\left(n-k\right)!\left(k+n-1\right)!}\ln\left(k\right).
\end{equation}

Since $\frac{\Gamma\left(z+a\right)}{\Gamma\left(z+b\right)}\sim z^{a-b}$ (if $z\to\infty$ in the sector $|\text{ph}\:z|\le \pi-\delta$ ) \cite{Asymptotic}, we have, for the absolute value of the summands in \ref{eq:32}:

$$
\left|\frac{\Gamma(x+n)\left(-1\right)^{n}}{\Gamma(x+1-n)\left(2n-1\right)}\sum_{k=1}^{n}\frac{\left(-1\right)^{k}\left(2k-1\right)\ln\left(k\right)}{\left(n-k\right)!\left(k+n-1\right)!}\right|\sim\left|\frac{x^{2n-1}}{2n-1}\sum_{k=1}^{n}\frac{\left(-1\right)^{k}\left(2k-1\right)\ln\left(k\right)}{\left(n-k\right)!\left(k+n-1\right)!}\right|
$$

$$
\le\left|\frac{x^{2n-1}}{2n-1}\sum_{k=1}^{n}\frac{2k^2}{\left(n-k\right)!\left(k-1\right)!}\right|=\left|\frac{x^{2n-1}}{\left(2n-1\right)}\frac{2^{n-2}n\left(n+3\right)}{\left(n-1\right)!}\right|.
$$

And so the $\Lambda$ function converges if $\sum_{n=1}^{\infty}\left|\frac{x^{2n-1}}{\left(2n-1\right)}\frac{2^{n-2}n\left(n+3\right)}{\left(n-1\right)!}\right|$ does. But since 
\begin{equation}\label{eq:1073}
\lim_{n\to\infty}\frac{2\left(n+1\right)\left(n+4\right)\left(2n-1\right)x^{2}}{n^{2}\left(n+3\right)\left(2n+1\right)}=0
\end{equation}
the series converges by the ratio test. Hence the $\Lambda$ function converges for the entire real axis. 
\end{proof}

\section{Inverse Gamma Function}
In this section, the series definition of the inverse gamma function given in conjecture \ref{conjecture} will be motivated. In particular, a Newton series will be constructed for this function, which we denote $\text{inv}\Gamma_0(x)$. For this purpose, it would be convenient to choose a series of the form 
\begin{equation}\label{eq:1074}
a_{1}+a_{2}\left(x-1!\right)+a_{3}\left(x-1!\right)\left(x-2!\right)+a_{4}\left(x-1!\right)\left(x-2!\right)\left(x-3!\right)+...
\end{equation}
i.e. selecting the nodes to be the points (1, 2), (2, 3), (6, 4), (24, 5),  etc. as the value of $\text{inv}\Gamma_0$ at these points is straightforward to compute. For comparison, it would be tricky to build a typical Newton series such that it interpolates $\text{inv}\Gamma_0(x)$ at integer values of $x$, as for example $\text{inv}\Gamma_0(3)=3.405869986...$ and $\text{inv}\Gamma_0(4)=3.664032797...$ have no known closed forms. 

With this in mind, we can begin computing the coefficients $a_n$ in equation \ref{eq:1074}: we first solve $x=1=\text{inv}\Gamma_0(1)$ to get $a_1=2$. Then we set $x=2!$ to get $a_{2}=1$, and $x=3!$ to get $a_3=-\frac{3}{20}$, and set $x=4!$ to get $a_4=\frac{559}{91080}$ etc. However, the resulting series diverges: as in the example from the introduction, the factorials grow too fast to be suitable as nodes for interpolation.

However, if computing the Newton series for the function $\ln\Gamma(y)=x$ using a series of the form 

$$a_{1}+a_{2}\left(x-\ln\left(1!\right)\right)+a_{3}\left(x-\ln\left(1!\right)\right)\left(x-\ln\left(2!\right)\right)+...$$
yields a series that grows at a slower rate. Setting $x=\ln(x)$ obtains a series for $\text{inv}\Gamma_0(x)$.

Performing the calculation manually yields the coefficients $a_{1}=2$, $a_{2}=\frac{1}{\ln\left(2\right)}$, $a_{3}=\frac{2}{\ln\left(6\right)\ln\left(3\right)}-\frac{1}{\ln\left(2\right)\ln\left(3\right)}$ etc.

Numerical tests indicate that the following formula

$$\sum_{k=0}^{n}k\left(\left(\prod_{i=1}^{k}\ln\left(\frac{\left(k+1\right)!}{i!}\right)\right)\prod_{i=1}^{n-k}\ln\left(\frac{\left(k+1\right)!}{\left(k+1+i\right)!}\right)\right)^{-1}$$
computes these coefficients. With this in mind, we propose the following:

\begin{conjecture}\label{conjecture}
The principal branch of the inverse gamma function $\Gamma(y)=x$ has the following Newton series representation

$$\text{inv}\Gamma_0(x)=2+\sum_{n=0}^{\infty}\sum_{k=0}^{n}k\left(\left(\prod_{i=1}^{k}\ln\left(\frac{\left(k+1\right)!}{i!}\right)\right)\prod_{i=1}^{n-k}\ln\left(\frac{\left(k+1\right)!}{\left(k+1+i\right)!}\right)\right)^{-1}\prod_{i=1}^{n}\ln\left(\frac{x}{i!}\right),$$

 which converges on the interval $(\Gamma(\alpha), \infty)$ where $\alpha=1.4616...$ is the unique positive number such that $\psi(\alpha)=0$.
\end{conjecture}

\appendix

\section{Appendix A}

\begin{lstlisting}[caption={Complex 3D plot of the $\Lambda$ function, made using Mathematica \cite{Wolfram}}
                \label{list:ajrec}]
  ComplexPlot3D[
  Product[
  Product[k^(((-1)^(n+k)(z+n-1)!(2k-1))/((n-k)!(n+k-1)!
  (z-n)!(2n-1))), {k,  1, n}], {n, 1, 10}]
  , {z, -4.5 - 2.5 I, 4.5 + 2.5 I}, Filling -> Bottom, PlotRange 
  -> {0, 6}, PlotPoints -> 50, Exclusions 
  -> None, Mesh -> {Join[10^Range[-4, 0, 0.5], Range[0, 6, 0.5]], 
  Range[-Pi, Pi, Pi/6]}, ViewPoint 
  -> {-.8, -3, 1}, Boxed -> False, AxesOrigin 
  -> {0, 0, 0}, AxesLabel -> {"x", "y", "z"}]
\end{lstlisting}


\vspace{6pt} 

\begin{thebibliography}{999}

\bibitem{Historical}
Davis, P. J. \textit{Leonhard Euler’s integral: A historical profile of the gamma function.} The American Mathematical Monthly 1959, 66, 862–865. 

\bibitem{Hermite}
Hermite, C. \textit{Extrait de Quelques Lettres de M. Ch. Hermite à M. S. Píncherle}. Annali di Matematica Pura ed Applicata 1901, 5, 57–72. 

\bibitem{notation}
Jeffrey, D. J.; Watt, Stephen M. \textit{The Inverse of the Complex Gamma Function}. ArXiv Preprint 2023.
\bibitem{1}
Borwein, J. M.; Corless, R. M. \textit{Gamma and factorial in the monthly}. The American Mathematical Monthly 2018, 125, 400–424. 

\bibitem{2}
Uchiyama, M. \textit{The principal inverse of the gamma function}. Proceedings of the American Mathematical Society 2012, 140, 1343–1348. 



\bibitem{3}
Pedersen, H. L. \textit{Inverses of gamma functions}. Constructive Approximation 2014, 41, 251–267. 

\bibitem{Norlund}
Nörlund, N. E. \textit{Sur les formules d’interpolation de Stirling et de Newton}. Annales scientifiques de l’École normale supérieure 1922, 39, 398–400. 

\bibitem{Laguerre}
Andrews, G. E.; Askey, R.; Roy, R. \textit{In Special functions}; Cambridge University Press: Cambridge, 1999; pp. 282–293. 

\bibitem{Bound}
Michalska, M.; Szynal, J. \textit{A new bound for the Laguerre polynomials}. Journal of Computational and Applied Mathematics 2001, 133, 489–493. 

\bibitem{Gauss}
Selby, S. M. \textit{In Standard mathematical tables and formulae}; CRC: Boca Raton, Florida, 1996; pp. 44–44. 

\bibitem{Stirling}
Abramowitz, M. (Ed); Stegun, I. A. (Ed) \textit{Handbook of Mathematical Functions, with formulas, graphs and mathematical tables}; 9th ed.; Dover Publications: New York, 1970, 824. 

\bibitem{Maple}
Johansson, F. \textit{Arbitrary-precision computation of the gamma function}. Maple Transactions 2023, 3. 

\bibitem{OEIS1}
OEIS Foundation Inc. (2023), Entry A360091 in The On-Line Encyclopedia of Integer Sequences, \url{http://oeis.org/A360091}


\bibitem{OEIS2}
OEIS Foundation Inc. (2023), Entry A360092 in The On-Line Encyclopedia of Integer Sequences, \url{http://oeis.org/A360092}


\bibitem{Stern}
Srinivasan, G. K. \textit{The Gamma Function: An eclectic tour.} The American Mathematical Monthly 2007, 114, 297–315. 

\bibitem{Wolfram}
Wolfram Research, Inc., Mathematica, Version 13.2, Champaign, IL 2022.


\bibitem{Asymptotic}
[DLMF] NIST Digital Library of Mathematical Functions.  Gamma Function, Properties, §5.11(iii) Ratios, F. W. J. Olver, A. B. Olde Daalhuis, D. W. Lozier, B. I. Schneider, R. F. Boisvert, C. W. Clark, B. R. Miller, B. V. Saunders, H. S. Cohl, and M. A. McClain, \url{https://dlmf.nist.gov/5.11}




\end{thebibliography}
\end{document}